\documentclass[11pt, reqno]{amsart} 
\usepackage{amsmath, amssymb, amsfonts, amstext, amsthm, amscd}

\usepackage[mathscr]{euscript}
\usepackage{enumerate}
\usepackage{tikz,soul}
\usepackage[english]{babel}
\usepackage{a4wide}
\usepackage{listings}
\usepackage{color} % Required for color definitions
\usepackage{xcolor} % Required for color definitions

\lstdefinelanguage{Macaulay2}{
  morekeywords={if, then, else, for, from, to, do, break, true, false, and, or, not, class, sum, all, apply, take, sort, print},
  sensitive=true,
  morecomment=[l]--,
  morestring=[b]",
}

\usepackage[left=2.5cm,top=2.5cm,bottom=3 cm,right=2.5cm]{geometry}
\usepackage{xcolor} % Required for specifying customized colours 
\colorlet{mdtRed}{red!50!black}
\definecolor{dblue}{rgb}{0,0,.6}
\usepackage[colorlinks]{hyperref}
\hypersetup{linkcolor=blue,citecolor=dblue,filecolor=dullmagenta,urlcolor=mdtRed}
\usepackage[capitalise]{cleveref}
\usetikzlibrary{arrows.meta, positioning,arrows}


\newcommand{\Fe}{\mathbb{F}_e}
\newcommand{\Fer}{\mathbb{F}_{e,r}}

\theoremstyle{plain}
\numberwithin{equation}{section}
\newtheorem{theorem}{Theorem}[section]
\newtheorem{proposition}[theorem]{Proposition}
\newtheorem{lemma}[theorem]{Lemma}
\newtheorem{corollary}[theorem]{Corollary}

\theoremstyle{definition}

\newtheorem{definition}[theorem]{Definition}
\newtheorem{remark}[theorem]{Remark}
\newtheorem{example}[theorem]{Example}

\newtheorem{Macau}[theorem]{Macaulay2 Code}
\numberwithin{equation}{section}
\title{Positivity of line bundles on general blow ups of Hirzebruch surfaces}

\author[C. J. Jacob]{Cyril J. Jacob}
	\address{Department of Mathematics, Indian Institute of Technology Bombay, Powai, Mumbai 
		400076, Maharashtra, India}
	\email{cyriljjacob@iitb.ac.in}
\author[B. Khan]{Bivas Khan}
\address{Department of Mathematical Sciences, Indian Institute of Technology (BHU), Varanasi 211005, India.}
\email{bivas.mat@iitbhu.ac.in}
\date{\today}
\begin{document}
	
	\begin{abstract}
		We investigate various positivity properties of line bundles on general blow ups of Hirzebruch surfaces motivated by \cite{Han}, where the author has studied general blow ups of $\mathbb{P}^2$. For each of the properties: ampleness, global generation, very ampleness, and $k$-very ampleness, we provide several sufficient numerical conditions. As an application, we obtain certain lower bounds for multi-point Seshadri constants of line bundles on the Hirzebruch surface. %for line bundles on general blow ups of Hirzebruch surfaces.
	\end{abstract}
	
	\maketitle

	\section{Introduction}\label{Introduction}
	
	Positivity properties of line bundles on blow ups of projective varieties have been of significant interest in the literature. There are several notions of the positivity of a line bundle, such as ampleness, global generation, very ampleness, and more generally $k$-very ampleness. We recall the definition of $k$-very ampleness.
	\begin{definition}
		A line bundle $L$ on a projective variety $X$ is said to be \textit{$k$-very ample} if the restriction map
		$$H^0(X,L)\to H^0(X,L\otimes\mathcal{O}_Z)$$
is surjective for all zero-dimensional subschemes $Z \subset X$ of length $k + 1$.
	\end{definition}
	Note that $0$-very ampleness and $1$-very ampleness of a line bundle are equivalent to the global generation and very ampleness, respectively. 
	
	Let $X$ be a surface and $p_1,p_2,\dots,p_r$ be $r$ distinct points in $X$. 
	Let $\pi: X_r \to X$ be the blow up at $p_1,p_2,\dots,p_r$. 
	Then 
	$$\text{Pic }X_r=\pi^\ast \text{Pic }X\oplus \mathbb{Z}.E_1 \oplus \mathbb{Z}.E_2 \oplus \dots \oplus \mathbb{Z}.E_r,$$
	where $E_i=\pi^{-1}(\{p_i\})$ is the exceptional divisor corresponding to $p_i$.
	So, a line bundle $L$ on $X_r$ has the form $L=\pi^\ast(\mathcal{L})-m_1E_1-\dots - m_rE_r$, where $\mathcal{L}\in \text{Pic }X$ and $m_1, \ldots, m_r \in \mathbb{Z}$. If $m_i=m$ for all $1\le i\le r$, we call $L$ a \textit{uniform line bundle}.

The ampleness of a line bundle can be characterized in different ways: geometrically (which we take as the definition), cohomologically (through the theorem of Cartan-Serre-Grothendieck), or numerically (using the Nakai-Moishezon-Kleiman criterion). Determining whether a given line bundle is ample is often challenging. Although the Nakai-Moishezon-Kleiman criterion provides a numerical condition, it involves an infinite number of conditions to check. So, it is natural to look for a finite number of sufficient conditions to check ampleness.
%So finding a finite set of sufficient conditions to check positivity of a line bundle efficiently embeds the variety to a projective space.
%and hence helps in better understanding of the variety.

In this article, we restrict our interest to surfaces.
It is known that if the \textit{Nagata-Biran-Szemberg conjecture} (see \cite[Section 4.1]{Sze}) holds for a surface $X$, then for a uniform line bundle $L$ on the blow up of $X$ at a sufficiently large number of very general points is ample if and only if $L^2>0$ (\cite[Theorem 3.1]{HJNS1}). Several finite conditions have been proposed for surfaces to determine the ampleness of a line bundle. The case of $ \mathbb{P}^2 $ has been extensively studied by various authors. Küchle (\cite{Ku}) and Xu (\cite{Xu}) independently showed that for the general blow up of $ \mathbb{P}^2$, the line bundle $L=\pi^\ast(\mathcal{O}_{\mathbb{P}^2}(d))-E_1-E_2-\dots-E_r$ is ample if $d\ge 3$ and $d^2>r$. The conditions for global generation and very ampleness can be found in \cite{Dav}. A complete characterization for very ampleness is given in \cite[Proposition 3.2]{GGP} when $ m_i = 1 $ for all $ i $. Later, Hanumanthu developed numerical criteria for the ampleness, global generation, and $ k $-very ampleness of line bundles, when blown up points are either general (\cite{Han}) or special (\cite{Han1}).

This question has also been explored for various other surfaces in addition to $ \mathbb{P}^2 $. For example, on an abelian surface, Lee and Shin provided criteria for ampleness and $ k $-very ampleness of uniform line bundles (see \cite{Lee}). For a ruled surface, Szemberg and Tutaj-Gasi\'nska provided a sufficient condition for a uniform line bundle with $ m = 1 $ to be $ k $-very ample when the blown-up points are general (\cite{ST}). In \cite{Far}, Farnik obtained similar results for $k$-very ampleness in the case of hyperelliptic surface.

% The study of blow ups of Hirzebruch surfaces is important because rational surfaces can be seen as blow ups of both $ \mathbb{P}^2 $ and Hirzebruch surfaces. 
In this article, we investigate the positivity of line bundles on the blow ups of Hirzebruch surfaces. Our primary motivation comes from \cite{Han}. 
%We address the topics of ampleness, global generation, very ampleness, and $ k $-very ampleness respectively, in each consecutive section. We have formulated a set of three theorems that provide sufficient conditions to establish the positivity of certain line bundles. 
In Section \ref{Ample}, we have obtained three sets of sufficient conditions for ampleness (see Theorem \ref{AmpleTh1}, Theorem \ref{AmpleTh2} and Theorem \ref{AmpTh3}). As an application, we have provided certain lower bounds for multi-point Seshadri constants of line bundles on the Hirzebruch surface (see Corollary \ref{multi-pt-SC}).
We have also obtained several classes of non-special divisors on blow up of Hirzebruch surfaces at $r$ very general points (see Corollary \ref{special}). We have also included a Macaulay2 package to check ampleness using these theorems.

In Section \ref{Global Generation}, we have obtained several sufficient conditions for global generation. The key tool here is the Reider's criterion (see Theorem \ref{Reider's}) along with the conditions for ampleness obtained in Section \ref{Ample}. In the final section, we have studied $k$-very ampleness. The key ingredient is a criterion by Beltrametti, Francia, and Sommese (see Theorem \ref{BFS}), along with the criteria for ampleness obtained in Section \ref{Ample}, as before.

Throughout this article, we work over the field $\mathbb{C}$ of complex numbers. 

	\subsection*{Acknowledgments} We sincerely thank Krishna Hanumanthu for suggesting the problem and for his valuable discussions and suggestions. We also thank Suhas B. N. for several helpful discussions. 

    \subsection*{AI disclosure} The authors acknowledge the use of ChatGPT for assistance with  proofreading. The author assumes full responsibility for the mathematical content, arguments, and claims presented in this work.

  %  \subsection*{Data availability statement}

%No data was used for the article.

%\subsection*{Conflict of interest statement}

%There is no conflict of interest.

    \subsection*{Competing interests} The authors declare no competing interests.

\subsection*{Funding declaration}

The authors are partially supported by a grant from the Infosys Foundation.

	\section{Preliminaries}
	Let $e\ge 0$ be an integer. Then the projective bundle associated with the rank two vector bundle $\mathcal{O}_{\mathbb{P}^1}\oplus\mathcal{O}_{\mathbb{P}^1}(-e)$ over $\mathbb{P}^1$,
	$$\Fe=\mathbb{P}(\mathcal{O}_{\mathbb{P}^1}\oplus\mathcal{O}_{\mathbb{P}^1}(-e))$$
	is called the \textit{Hirzebruch surface} with invariant $e$. Let $\phi: \Fe \to \mathbb{P}^1$ denote the canonical projection map. Let $C_e$ denote the image of a section of $\phi: \Fe \to \mathbb{P}^1$ such that $C_e$ is the
	divisor associated to the line bundle $\mathcal{O}_{\Fe}(1)$. Also, let $f$ denote a fiber of the map $\phi : \Fe \to \mathbb{P}^1$. Then we have
	$$\text{Pic } \Fe = \mathbb{Z}.C_e \oplus \mathbb{Z}.f,$$
	with  the following intersection products
    \begin{equation*}
         C_e^2=-e, \, f^2=0 \text{ and } C_e\cdot f=1.
    \end{equation*}
   %\textcolor{blue}{defining intersecttion product on Pic $\Fe$}.

   We are interested in the blow up of $\Fe$ at finitely many distinct points. Let $\pi: \Fer \to \Fe$ denote the blow up of $\Fe$ at $r$ distinct points $p_1,p_2,\dots,p_r$. Let $H_e = \pi^*(C_e)$, $F_e = \pi^*(f)$, and $E_i=\pi^{-1}(\{p_i\})$ is the exceptional divisor corresponding to $p_i$. We know that
	$$\text{Pic } \Fer = \mathbb{Z}.H_e \oplus \mathbb{Z}.F_e \oplus \mathbb{Z}.E_1 \oplus \mathbb{Z}.E_2 \oplus \dots \oplus \mathbb{Z}.E_r.$$
	 Moreover, in both $\mathbb{F}_e$ and $\mathbb{F}_{e,r}$, the numerical equivalence coincides with the linear equivalence. Let $K_{\mathbb{F}_e}$ and $K_{\mathbb{F}_{e,r}}$ denote the canonical line bundles of $\mathbb{F}_e$ and
        $\mathbb{F}_{e,r}$, respectively. Then, we have $$ K_{\mathbb{F}_e} = -2 C_e - (e+2) f $$ and $$K_{\mathbb{F}_{e,r}} = -2 H_e - (e+2) F_e + \sum\limits_{i=1}^r E_i.$$
         By abuse of notation, we will use $H_e$ and $F_e$ to refer both the divisor classes and a specific curve linearly equivalent to these divisor classes. 	
	
% \begin{lemma}\label{pfw_phi}
% 		Let $\phi:\Fe\to \mathbb{P}^1$ denote the Hirzebruch surface with invariant $e$. Then for \(a \geq 0\), we have
% 		$$\phi_\ast(aC_e)=\mathcal{O}_{\mathbb{P}^1}\oplus\mathcal{O}_{\mathbb{P}^1}(-e)\oplus\mathcal{O}_{\mathbb{P}^1}(-2e)\oplus\dots \oplus \mathcal{O}_{\mathbb{P}^1}(-ae).$$
% 	\end{lemma}
% 	\begin{proof}
% 		From \cite[Proposition II.7.11]{Har} we can see that 
% 		\begin{equation}
% 			\begin{split}
% 				\phi_\ast(aC_e)&=\textbf{\textit{S}}^a(\mathcal{O}_{\mathbb{P}^1}\oplus\mathcal{O}_{\mathbb{P}^1}(-e))\\
% 				&=\bigoplus_{i+j=a}\left(\textbf{\textit{S}}^i(\mathcal{O}_{\mathbb{P}^1})\otimes\textbf{\textit{S}}^j(\mathcal{O}_{\mathbb{P}^1}(-e))\right) ~(\text{by \cite[Ex. II.5.16(c) ]{Har}})\\
% 				&=\mathcal{O}_{\mathbb{P}^1}\oplus\mathcal{O}_{\mathbb{P}^1}(-e)\oplus\mathcal{O}_{\mathbb{P}^1}(-2e)\oplus\dots \oplus \mathcal{O}_{\mathbb{P}^1}(-ae).
% 			\end{split}
% 		\end{equation}
% 	\end{proof}
	
	\begin{proposition}\cite[ Proposition 2.8]{JKS}\label{dim}
		Let $\phi:\Fe\to \mathbb{P}^1$ denote the Hirzebruch surface with invariant $e$.  Then we have
		$$h^0(\Fe,aC_e+bf)=(a+1)\left(b+1-\frac{ae}{2}\right),$$
		where $ a \geq 0 \text{ and }b \geq ae$.
	\end{proposition}
	% \begin{proof}
	% 	It is easy to see that $$h^0(\Fe,aC_e+bf)=h^0(\mathbb{P}^1,\phi_\ast(aC_e+bf)).$$
	% 	Then by the projection formula, we have
	% 	\[h^0(\Fe,aC_e+bf)=h^0(\mathbb{P}^1,\phi_\ast(aC_e)\otimes \mathcal{O}_{\mathbb{P}^1}(b)).\]
	% 	Now by Lemma \ref{pfw_phi}, we get
	% 	\begin{equation*}
	% 		\begin{split}
	% 			h^0(\Fe,aC_e+bf)&=h^0(\mathbb{P}^1,\phi_\ast(aC_e)\otimes \mathcal{O}_{\mathbb{P}^1}(b))\\
	% 		%	&=h^0(\mathbb{P}^1,(\mathcal{O}_{\mathbb{P}^1}\oplus\mathcal{O}_{\mathbb{P}^1}(-e)\oplus\mathcal{O}_{\mathbb{P}^1}(-2e)\oplus\dots \oplus \mathcal{O}_{\mathbb{P}^1}(-ae)e))\otimes \mathcal{O}_{\mathbb{P}^1}(b))\\
	% 			&=h^0(\mathbb{P}^1,\mathcal{O}_{\mathbb{P}^1}(b)\oplus\mathcal{O}_{\mathbb{P}^1}(b-e)\oplus\mathcal{O}_{\mathbb{P}^1}(b-2e)\oplus\dots \oplus \mathcal{O}_{\mathbb{P}^1}(b-ae))\\
	% 			&=h^0(\mathbb{P}^1,\mathcal{O}_{\mathbb{P}^1}(b))+h^0(\mathbb{P}^1,\mathcal{O}_{\mathbb{P}^1}(b-e))+\dots+h^0(\mathbb{P}^1,\mathcal{O}_{\mathbb{P}^1}(b-ae))\\
	% 			&=b+1+b-e+1+\dots+b-ae+1\\
	% 			&=(a+1)b+a+1-\frac{a(a+1)}{2}e\\
	% 			&=(a+1)(b+1)-\frac{a(a+1)}{2}e\\
	% 			&=(a+1)(b+1-\frac{ae}{2}).
	% 		\end{split}		
	% 	\end{equation*}
		
	% \end{proof}

The Nakai-Moishezon criterion provides a numerical criterion for determining the ampleness of a line bundle. Analogously, the following two theorems establish sufficient numerical conditions for global generation, very ampleness, and $k$-very ampleness.
	%The following theorem is the key element that enables us to establish sufficient conditions for global generation and very ampleness.
	\begin{theorem}[Reider's criterion]\cite[Theorem 1]{Re}\label{Reider's}
	Let $X$ be a smooth projective surface and let $N$ be
	a nef line bundle on $X$.

	If $N^2\ge5$ and $K_X+N$ is not globally generated,
	then there exists a nonzero effective divisor $D$ on
	$X$ such that one of the following holds:
	\begin{enumerate}
		\item $N\cdot D=0$ and $D^2=-1$;
		\item $N\cdot D=1$ and $D^2=0$.
	\end{enumerate}

	If $N^2\ge10$ and $K_X+N$ is not very ample, then
	there exists a nonzero effective divisor $D$ on $X$
	such that one of the following holds:
	\begin{enumerate}
		\item $N\cdot D=0$ and $D^2=-1$ or $-2$;
		\item $N\cdot D=1$ and $D^2=0$ or $-1$;
		\item $N\cdot D=2$ and $D^2=0$.
	\end{enumerate}
\end{theorem}

\begin{theorem}[Beltrametti-Francia-Sommese]\label{BFS}\cite[Theorem 2.1]{Bfs}
	Let $X$ be a smooth projective surface, let $k\ge0$
	be an integer, and let $N$ be a nef line bundle on
	$X$ with $N^2\ge4k+5$.
	If $K_X+N$ is not $k$-very ample, then there exists
	a nonzero effective divisor $D$ on $X$ such that
	$$
	N\cdot D-k-1\le D^2
	<\frac{N\cdot D}{2}<k+1.
	$$
\end{theorem}
	\section{Ampleness}\label{Ample}
	Let $\pi: \Fer \to \Fe$ denote the blow up of $\Fe$ at $r$ distinct points $p_1,p_2,\dots,p_r$. In this section, we establish several criteria for determining the ampleness of a given line bundle on $\Fer$. Recently, in \cite[Proposition 3.11]{HJNS}, the authors have obtained a necessary and sufficient condition for ampleness when the number of blown up points is less than or equal to $e+1$. Here, we provide several sufficient conditions for ampleness which will work for any the number of blown up points.
	\begin{theorem}\label{AmpleTh1}
		Suppose $e>0$. Let $p_1,p_2,\dots,p_r$ be $r$ distinct points on $\Fe$ such that for each $i$, $p_i \notin C_e$, and for $i, j \in \{1, \ldots , r\}$ with $i \neq j$, $p_i$ and $p_j$ are not on the same fiber of the map $\phi : \Fe \to \mathbb{P}^1$. Let $\pi : \Fer \to \Fe$ be the blow up of $\Fe$ at $p_1,p_2,\dots,p_r$. Let $L=aH_e+bF_e-m_1E_1-m_2E_2-\dots-m_rE_r$ be a line bundle on $\Fer$ such that:
		\begin{enumerate}
			\item \label{A1C1} $a>m_i>0$ ~$\text{for } 1\le i\le r,$
			\item \label{A1C2} $b>ae,$
			\item \label{A1C3} $b>\displaystyle \sum_{i=1}^{r}m_i.$
		\end{enumerate}
		Then $L$ is ample.
		\begin{proof}
				Using Nakai-Moishezon criteria for ampleness, it is enough to show that $L^2>0$ and $L\cdot C>0$ for all reduced irreducible curves on $\Fer$. First, we will show that $L^2>0$. Note that from $(1)$ and $(3)$, we have:
                \begin{equation}\label{Eqnmini-1}
                    ab>\displaystyle a\sum_{i=1}^{r}m_i>\sum_{i=1}^{r}m_i^2.
                \end{equation}
			Then, we have
			\begin{equation*}
				\begin{split}
					L^2&=2ab-a^2e-\sum_{i=1}^{r}m_i^2=ab+a(b-ae)-\sum_{i=1}^{r}m_i^2\\
							& >0 \, (\text{by $\eqref{Eqnmini-1}$, \eqref{A1C1} and \eqref{A1C2}}).
				\end{split}
			\end{equation*} 
			Now, we show that $L\cdot C>0$ for any reduced irreducible curve $C$ on $\Fer$.
			Write $$C=\alpha H_e+\beta F_e-n_1E_1-n_2E_2-\dots-n_rE_r,$$
           where $\alpha, \beta, n_1, \ldots, n_r \in \mathbb{Z}$. Now if $\alpha=0$ or $\beta=0$, from \cite[Corollary V.2.18]{Har} we can see that $C$ is a curve in the set $\{ E_i,H_e,F_e,F_e-E_i\}$ for some $i=1, \ldots,  r$. Hence, from  \eqref{A1C1} and  \eqref{A1C2}, $L\cdot C>0$.
			
			Now suppose that $\alpha \neq 0, \beta \neq 0$. Then again from \cite[Corollary V.2.18]{Har}, we have $\beta\ge \alpha e$. As $C$ and $F_e-E_i$ are distinct curves, using B\'ezout's theorem we have
            $$\alpha \ge n_i \text{ for all } i=1, \ldots, r.$$
            Hence, by \eqref{A1C3} we have, 
			\begin{equation}\label{Eqnmini}
				\alpha b > \sum_{i=1}^rm_in_i.
			\end{equation}
			Thus, we get
			\begin{equation*}
				\begin{split}
					L\cdot C&= a\beta+\alpha b -a\alpha e -\sum_{i=1}^{r}m_in_i = a(\beta-\alpha e)+(\alpha b -\sum_{i=1}^r m_in_i)>0 ~(\text{by  \eqref{Eqnmini}}).
				\end{split}
			\end{equation*}
			Hence $L$ is ample.
		\end{proof}
 		
	\end{theorem}
	
	In \cite[Lemma 3.2]{HJNS}, the authors have obtained a bound for the sum of multiplicity of an irreducible and reduced curve at $r$ distinct points when $r\le e+1$. The following lemma provides a bound for any $r\in \mathbb{N}$.

	\begin{lemma}\label{multi}
	 Let $n\ge e$ be an integer. Let $C=\alpha C_e+\beta f$ be a reduced and irreducible curve on $\Fe$ with $\alpha>1$ or $\beta >n$. Then for any distinct points $p_1,p_2,\dots,p_{2n-e+1}$ in $\Fe$, we have
	\begin{equation*}
	    \beta +(n-e)\alpha \ge \sum_{i=1}^{2n-e+1}\text{\rm mult}_{p_i}C.
	\end{equation*}
    \end{lemma}
	\begin{proof}
		From  Proposition \ref{dim}, we have $h^0(C_e+nf)=2n-e+2$. So there exists a curve $D\in |C_e+nf|$ passing through $p_1,p_2,\dots,p_{2n-e+1}$.  Since $C$ is reduced and irreducible, $C$ and $D$ can have a common component only if $C$ is a component of $D$. But
		$D-C$ can not be effective as $D-C\equiv (1-\alpha)C_e+(n-\beta)f$ and either $1-\alpha$ or $n-\beta$ is negative by the assumption on $C$.  Hence, $C$ is not a component of $D$.

		So, by Bez\'out's theorem, we have
		$$C\cdot D\ge \sum_{i=1}^{2n-e+1}\text{mult}_{p_i}C~\text{mult}_{p_i}D\ge \sum_{i=1}^{2n-e+1}\text{mult}_{p_i}C.$$
	\end{proof}

	We formulate the following sufficient conditions for ampleness.
	\begin{theorem}\label{AmpleTh2}
    Suppose $e>0$. Let $\pi : \Fer \to \Fe$ be the blow up of
    $\Fe$ at very general points $p_1,p_2,\dots,p_r$, where $r\ge1$.
    Let $\lambda\ge e$ be the least integer such that
    $2\lambda-e+2>r$. Also, let $k_i=2i-e+1$ and
    $\ell_i=\min\{k_i,r\}$ for all $e\le i\le\lambda$.
    Let
    $L=aH_e+bF_e-m_1E_1-m_2E_2-\dots-m_rE_r$
    be a line bundle on $\Fer$ such that:
    \begin{enumerate}
        \item \label{A2C1}
        $a>m_i>0$, for $1\le i\le r$,
        
        \item \label{A2C2}
        $\displaystyle b+(i-e)a>\sum_{j\in I}m_j$,
        for every $e\le i\le\lambda$ and every subset
        $I\subseteq\{1,\dots,r\}$ of cardinality $\ell_i$,
        
        \item \label{A2C3}
        $b>a\lambda$.
    \end{enumerate}
    Then $L$ is ample.
    
    \begin{proof}
        Using the Nakai-Moishezon criterion, it is enough to prove
        that $L^2>0$ and $L\cdot C>0$ for every reduced irreducible
        curve $C$ on $\Fer$.
        
        By \eqref{A2C1} and \eqref{A2C2}, we have
        \begin{equation}\label{EqnL2}
            a\bigl(b+(i-e)a\bigr)
            >a\sum_{j\in I}m_j
            >\sum_{j\in I}m_j^2,
            \qquad |I|=\ell_i,\quad e\le i\le\lambda.
        \end{equation}
        Since $k_\lambda\ge r$, we have $\ell_\lambda=r$.
        Therefore,
        \[
        \begin{split}
            L^2
            &=2ab-a^2e-\sum_{i=1}^r m_i^2\\
            &=a(b-a\lambda)
              +a\bigl(b+a(\lambda-e)\bigr)
              -\sum_{i=1}^r m_i^2\\
            &>0
            \quad\text{(by \eqref{A2C1}, \eqref{A2C3}
            and \eqref{EqnL2}).}
        \end{split}
        \]
        
        Now, let $C$ be a reduced irreducible curve on $\Fer$.
        We prove that $L\cdot C>0$ by considering the following cases.
        
        \noindent\textbf{Case (i):} $C=E_i$ for some $1\le i\le r$.\\
        Then $L\cdot C=m_i>0$ by \eqref{A2C1}.
        
        For the remaining cases, $C$ is the strict transform of
        $\pi(C)$. Write
        \[
            C=\alpha H_e+\beta F_e-\sum_{i=1}^r n_iE_i,
            \qquad n_i=\operatorname{mult}_{p_i}\pi(C)\ge0.
        \]
        
        \noindent\textbf{Case (ii):} $\pi(C)=C_e$.\\
        Since the points are very general, none lies on $C_e$.
        Hence $C=H_e$, and
        \[
            L\cdot C=b-ae\ge b-a\lambda>0
            \quad\text{by \eqref{A2C3}.}
        \]
        
        \noindent\textbf{Case (iii):} $\pi(C)$ is a fiber.\\
        Since the points are very general, at most one lies on
        any given fiber. Hence $C=F_e$ or $C=F_e-E_i$ for some $i$.
        Thus
        \[
            L\cdot F_e=a>0,\qquad
            L\cdot(F_e-E_i)=a-m_i>0
        \]
        by \eqref{A2C1}.
        
        For the remaining curves, using
        \cite[Corollary V.2.18]{Har}, we have
        $\alpha>0$ and $\beta\ge\alpha e$.
        Intersecting $\pi(C)$ with the fiber through $p_i$ gives
        $n_i\le\alpha$ for every $i$.
        
        \noindent\textbf{Case (iv):}
        $\alpha=1$ and $e\le\beta\le\lambda$.\\
        Since $0\le n_i\le1$, each $n_i$ is either $0$ or $1$.
        By Proposition \ref{dim},
        \[
            h^0(\Fe,\mathcal{O}_{\Fe}(C_e+\beta f))
            =2\beta-e+2.
        \]
        Since the points are very general, at most
        $2\beta-e+1=k_\beta$ of them lie on $\pi(C)$.
        Thus, if $J=\{i\mid n_i=1\}$, then
        $|J|\le\min\{k_\beta,r\}=\ell_\beta$.
        Choose a subset $I\subseteq\{1,\dots,r\}$ of cardinality
        $\ell_\beta$ containing $J$. By \eqref{A2C1}
        and \eqref{A2C2}, we obtain
        \begin{equation}\label{Eqn2}
            b+(\beta-e)a
            >\sum_{j\in I}m_j
            \ge\sum_{j\in J}m_j
            =\sum_{i=1}^r m_in_i.
        \end{equation}
        Consequently,
        \[
            L\cdot C
            =b+(\beta-e)a-\sum_{i=1}^r m_in_i
            >0
            \quad\text{by \eqref{Eqn2}.}
        \]
        
        \noindent\textbf{Case (v):}
        $\alpha>1$ or $\beta>\lambda$.\\
        By Lemma \ref{multi}, we have
        \[
            \beta+\alpha(\lambda-e)\ge\sum_{i=1}^r n_i.
        \]
        %Here, if $r<k_\lambda$, we apply the lemma after adding
        %$k_\lambda-r$ distinct auxiliary points and discard
        %their nonnegative multiplicities.
        Combining this inequality with \eqref{A2C1}, we get
        \begin{equation}\label{Eqn3}
            a\bigl(\beta+\alpha(\lambda-e)\bigr)
            \ge a\sum_{i=1}^r n_i
            \ge\sum_{i=1}^r m_in_i.
        \end{equation}
        Therefore,
        \[
        \begin{split}
            L\cdot C
            &=a\beta+\alpha b-a\alpha e
              -\sum_{i=1}^r m_in_i\\
            &=a\bigl(\beta+\alpha(\lambda-e)\bigr)
              +\alpha(b-a\lambda)
              -\sum_{i=1}^r m_in_i\\
            &\ge\alpha(b-a\lambda)
              \quad\text{(by \eqref{Eqn3})}\\
            &>0
              \quad\text{(by \eqref{A2C3} and $\alpha>0$).}
        \end{split}
        \]
        Hence $L$ is ample.
    \end{proof}
\end{theorem}
	 
	\begin{remark}
		The above proof shows that conditions \eqref{A2C1} and \eqref{A2C2} are necessary for ampleness. 
        %Whereas \eqref{A2C3} is not a necessary condition to have.
	\end{remark}
	The following theorem is motivated from \cite[Theorem 2.1]{Han}, where the author has obtained criterion for ampleness of line bundles on the blow up of $\mathbb{P}^2$.
	\begin{theorem}\label{AmpTh3}
	Suppose $e>0$. Let $\pi : \Fer \to \Fe$ be the blow up of $\Fe$ at very general points $p_1,p_2,\dots,p_r$, where $r \geq 1$. Let $L=aH_e+bF_e-m_1E_1-m_2E_2-\dots-m_rE_r$ be a line bundle on $\Fer$ such that:
	\begin{enumerate}
		\item \label{A3C1} $a>m_i>0$, ~~$\forall 1\le i\le r,$
		\item \label{A3C2} $b>ae,$
		\item \label{A3C4} $\displaystyle 2ab-a^2e >
		\frac{s}{s-1}\sum_{j=1}^s m_{i_j}^2,$
		for every $2\le s\le r$ and every choice of distinct indices
		$i_1,i_2,\dots,i_s\in\{1,\dots,r\}$.
	\end{enumerate}
	Then $L$ is ample.
	\begin{proof}If $r=1$, then $b>ae\ge a>m_1>0$, so $L$ is ample
by Theorem \ref{AmpleTh1}. Hence we may assume that $r\ge2$.

		Using the Nakai-Moishezon criterion, it is enough to prove that
		$L^2>0$ and $L\cdot C>0$ for all reduced irreducible curves
		$C$ on $\Fer$. Since $\frac{r}{r-1}>1$, using
		\eqref{A3C4} we have,
		$$
		L^2=2ab-a^2e-\sum_{i=1}^r m_i^2>0.
		$$
		Now, let $C$ be a reduced irreducible curve on $\Fer$.
		If $C=E_i$ for some $1\le i\le r$, then clearly
		$L\cdot C=m_i>0$ by \eqref{A3C1}.
		Otherwise, $C$ is the strict transform of $\pi(C)$.
		Let $n_i=\text{mult}_{p_i}\pi(C)$, so we can write
		$$
		C=\alpha H_e+\beta F_e-n_1E_1-n_2E_2-\dots-n_rE_r.
		$$
		If $\pi(C)=C_e$, then $C=H_e$, since the blown up points
		are very general. Hence $L\cdot C=b-ae>0$ by
		\eqref{A3C2}. If $\pi(C)$ is a fiber, then $C=F_e$ or
		$C=F_e-E_i$ for some $i$, and $L\cdot C>0$ by
		\eqref{A3C1}.

		Thus, using \cite[Corollary V.2.18]{Har}, we may assume
		that $\alpha>0$ and $\beta\ge\alpha e$.
		Intersecting $\pi(C)$ with the fiber through $p_i$ gives
		$n_i\le\alpha$ for every $i$. Also,
		$$
		a\beta+b\alpha-a\alpha e
		=a(\beta-\alpha e)+b\alpha>0.
		$$
		Let $s$ denote the cardinality of the set
		$\{i\mid n_i>0,\ 1\le i\le r\}$.
		We consider the following cases.\\

		\noindent\textbf{Case (i):} $s\le 1$.\\
		If $s=0$, then
		$$
		L\cdot C=a(\beta-\alpha e)+b\alpha>0.
		$$
		If $s=1$, let $n_{i_1}$ be the only positive multiplicity.
		Since $b>ae\ge a>m_{i_1}$ and $\alpha\ge n_{i_1}$,
		we have
		$$
		L\cdot C
		=a(\beta-\alpha e)+b\alpha-m_{i_1}n_{i_1}
		\ge b\alpha-m_{i_1}n_{i_1}>0.
		$$

		\noindent\textbf{Case (ii):} $s\ge 2$.\\
		Let $i_1,i_2,\dots,i_s$ be the indices for which
		$n_i>0$, and let $n=\min\{n_{i_1},\dots,n_{i_s}\}$.
		It is easy to see that,
		\begin{equation}\label{Eqn4}
			\begin{split}
				(a\beta+b\alpha-a\alpha e)^2
				&=(2ab-a^2e)(2\alpha\beta-\alpha^2e)
				+(a\beta-b\alpha)^2\\
				&\ge(2ab-a^2e)(2\alpha\beta-\alpha^2e).
			\end{split}
		\end{equation}
		By \cite[Lemma 2.2]{HJNS}, we have
		$$
		2\alpha\beta-\alpha^2e
		\ge\sum_{j=1}^s n_{i_j}^2-n.
		$$
		Since $n_{i_j}\ge n\ge1$ for every $j$, we have
		$\sum_{j=1}^s n_{i_j}^2\ge sn$. Hence,
		\begin{equation}\label{Eqn-1}
			\sum_{j=1}^s n_{i_j}^2-n
			\ge\frac{s-1}{s}\sum_{j=1}^s n_{i_j}^2>0.
		\end{equation}
		Combining these inequalities with \eqref{A3C4} and
		\eqref{Eqn4}, we get
		\begin{equation}\label{Eqn-2}
			\begin{split}
				(a\beta+b\alpha-a\alpha e)^2
				&\ge(2ab-a^2e)(2\alpha\beta-\alpha^2e)\\
				&\ge(2ab-a^2e)
				\left(\sum_{j=1}^s n_{i_j}^2-n\right)\\
				&>\frac{s}{s-1}
				\left(\sum_{j=1}^s m_{i_j}^2\right)
				\left(\sum_{j=1}^s n_{i_j}^2-n\right)
				\quad\text{(by \eqref{A3C4})}\\
				&\ge
				\left(\sum_{j=1}^s m_{i_j}^2\right)
				\left(\sum_{j=1}^s n_{i_j}^2\right)
				\quad\text{(by \eqref{Eqn-1})}\\
				&\ge
				\left(\sum_{j=1}^s m_{i_j}n_{i_j}\right)^2
				\quad\text{(by Cauchy-Schwarz's inequality).}
			\end{split}
		\end{equation}
		Since $a\beta+b\alpha-a\alpha e>0$, it follows that
		$$
		L\cdot C
		=a\beta+b\alpha-a\alpha e
		-\sum_{j=1}^s m_{i_j}n_{i_j}>0.
		$$
		Thus $L$ is ample.
	\end{proof}
\end{theorem}
	% \begin{remark}
	% 	Note that for a line bundle $L$ on $\Fer$ with $L^2>0$, we get that $L\cdot C>0$ for all curves $C$ with $C^2\ge0$. Because, we have
 %        \begin{equation*}
 %            \begin{split}
 %                (L \cdot C+ \sum_{i=1}^r m_in_i)^2 &=(\alpha b + a \beta - a \alpha e)^2\\
 %                & \geq (2 ab - a^2 e)(2 \alpha \beta - \alpha^2 e) \text{ (using \eqref{Eqn4})}\\
 %                & > \left(\sum_{i=1}^r m_i\right)^2\left(\sum_{i=1}^r n_i\right)^2 \text{ (as $L^2 >0$ and $C^2 \ge 0$)}\\
 %                & \geq \left(\sum_{i=1}^r m_in_i\right)^2 \text{ (using Cauchy–Schwarz inequality)}.
 %            \end{split}
 %        \end{equation*}
 %        Hence, if a line bundle $L$ with $L^2>0$ is not ample, there exists a curve $C$ with $C^2<0$ such that $L\cdot C<0$.
	% \end{remark}
    
	The following example shows that the set of line bundles satisfying Theorems \ref{AmpleTh1}, \ref{AmpleTh2} and \ref{AmpTh3} are different. 
	\begin{example}
		Consider the surface $\mathbb{F}_{10,12}$. Note that the line bundle $L=33H_e+331F_e-32E_1-32E_2-\dots-32E_{10}-E_{11}-E_{12}$ is ample using Theorem \ref{AmpleTh1}. But $L$ does not satisfy the conditions in Theorem \ref{AmpleTh2} and Theorem \ref{AmpTh3}. 
        
        Similarly, $32H_e+353F_e-31E_1-\dots-31E_{12}$ is an ample line bundle by Theorem \ref{AmpleTh2}, which does not satisfy the conditions of the other two theorems.

        Finally, consider the  line bundle $L:=2H_e+6F_e-E_1-E_2-\dots-E_{10}$ on $\mathbb{F}_{2,10}$, which  is ample by Theorem \ref{AmpTh3}. Note that $L$ does not satisfy the conditions in either Theorem \ref{AmpleTh1} or Theorem \ref{AmpleTh2}.
	\end{example}

    \begin{Macau}
    The following Macaulay2 code checks the sufficient numerical
    conditions for ampleness in Theorem \ref{AmpleTh1},
    Theorem \ref{AmpleTh2} and Theorem \ref{AmpTh3}.
    We assume that $e>0$, $r\ge1$, and the blown up points are
    very general. The inputs $e,a,b,m_1,\dots,m_r$ must be integers.

    Since the multiplicities must be positive, sorting them in
    decreasing order allows us to check the subset conditions
    using only the largest multiplicities. The function returns
    \texttt{true} if at least one theorem applies. A return value
    of \texttt{false} means that these criteria give no conclusion;
    it does not imply that the line bundle is not ample.

    \begin{lstlisting}[
        caption={Macaulay2 function \texttt{checkAmplenessAll}},
        label={lst:checkAmpleness}
    ]
checkAmplenessAll = (e, a, b, mList) -> (
    -- Validate the inputs
    if not (class e === ZZ and class a === ZZ
        and class b === ZZ) then
        error "e, a, and b must be integers.";

    if not (class mList === List) then
        error "mList must be a list.";

    if not all(mList, mi -> class mi === ZZ) then
        error "All multiplicities must be integers.";

    if e <= 0 then error "The criteria require e > 0.";

    r := #mList;
    if r == 0 then error "At least one point is required.";

    mSorted := rsort mList;
    sumM := sum mSorted;

    commonCond := all(mSorted,
        mi -> a > mi and mi > 0);

    -- Theorem 3.1
    t31 := commonCond and b > a*e and b > sumM;

    if t31 then
        print "L is ample by Theorem 3.1."
    else
        print "No conclusion by Theorem 3.1.";

    -- Theorem 3.3
    lambda := e;
    while (2*lambda - e + 2) <= r do
        lambda = lambda + 1;

    subsetCond := true;
    for i from e to lambda do (
        ki := 2*i - e + 1;
        ell := min(ki, r);
        sumK := sum(take(mSorted, ell));

        if b + (i-e)*a <= sumK then
            subsetCond = false;
    );

    t33 := commonCond and subsetCond and b > a*lambda;

    if t33 then
        print "L is ample by Theorem 3.3."
    else
        print "No conclusion by Theorem 3.3.";

    -- Theorem 3.5
    q := 2*a*b - a^2*e;
    squareCond := true;

    if r >= 2 then (
        for s from 2 to r do (
            sumSq := sum(apply(take(mSorted, s),
                mi -> mi^2));

            -- Equivalent to q > s*sumSq/(s-1)
            if (s-1)*q <= s*sumSq then
                squareCond = false;
        );
    );

    t35 := (r >= 2) and commonCond
        and b > a*e and squareCond;

    if t35 then
        print "L is ample by Theorem 3.5."
    else
        print "No conclusion by Theorem 3.5.";

    -- Final result
    if t31 or t33 or t35 then (
        print "Final verdict: L is ample.";
        true
    ) else (
        print "Final verdict: these criteria give no conclusion.";
        false
    )
);
    \end{lstlisting}
\end{Macau}
\subsection{Application to bounds on multi-point Seshadri constants}

The multi-point Seshadri constant of an ample line bundle
$\mathcal{L}$ on a smooth projective variety $X$ at distinct
points $p_1,p_2,\dots,p_r$ is the real number
$$
\varepsilon(X,\mathcal{L},p_1,\dots,p_r)
:=\sup\left\{
m\in\mathbb{R}_{\ge0}\,:\,
\pi^*\mathcal{L}-m\sum_{i=1}^r E_i
\text{ is ample}
\right\},
$$
where $\pi:\Tilde{X}\to X$ is the blow up of $X$ at
$p_1,p_2,\dots,p_r$ (see \cite[Definition 1.5 and 1.8]{Bau}).

%Let $\mathcal{L}=aC_e+bf$ be an ample line bundle on the Hirzebruch surface $\Fe$. 
Then we have the following lower
bounds.

\begin{corollary}\label{multi-pt-SC}
	Suppose $e>0$. Let $\mathcal{L}=aC_e+bf$ be an ample line bundle on the
Hirzebruch surface $\Fe$. Let $\pi:\Fer\to\Fe$ be the blow up of
	$\Fe$ at $r$ distinct points $p_1,p_2,\dots,p_r$.
	\begin{enumerate}
		\item \label{SC1} If $r\ge2$ and the points $p_1,p_2,\dots,p_r$
		are in very general position, then
		$$
		\varepsilon(\Fe,\mathcal{L},p_1,\dots,p_r)
		\ge
		\min\left\{
			\sqrt{\frac{r-1}{r}}
			\sqrt{\frac{\mathcal{L}^2}{r}},\,a
		\right\}.
		$$

		\item \label{SC2}If $p_i\notin C_e$ for every $i$, and no two
		of the points $p_1,p_2,\dots,p_r$ lie on the same
		fiber of $\phi:\Fe\to\mathbb{P}^1$, then
		$$
		\varepsilon(\Fe,\mathcal{L},p_1,\dots,p_r)
		\ge\min\left\{\frac{b}{r},\,a\right\}.
		$$

		\item \label{SC3} Suppose that the points $p_1,p_2,\dots,p_r$
		are in very general position. Let $\lambda\ge e$
		be the least integer such that $2\lambda-e+2>r$,
		and let $k_i=2i-e+1$ for all $e\le i\le\lambda$.
		If $b>a\lambda$, then
		$$
		\varepsilon(\Fe,\mathcal{L},p_1,\dots,p_r)
		\ge
		\min\left\{
			\min_{e\le i\le\lambda}
			\frac{b+(i-e)a}{k_i},\,a
		\right\}.
		$$
	\end{enumerate}

	\begin{proof}
		Since $\mathcal{L}$ is ample, we have $a>0$ and
		$b>ae$. For a positive rational number $m$, write
		$$
		L_m=\pi^*\mathcal{L}-m\sum_{i=1}^r E_i.
		$$
		We apply the ampleness criteria to an integral
		multiple of $L_m$ after clearing denominators.

		For (\ref{SC1}), let
		$$
		0<m<
		\min\left\{
			\frac{\sqrt{(r-1)\mathcal{L}^2}}{r},\,a
		\right\}.
		$$
		The function $\frac{s^2}{s-1}$ is increasing for
		integers $s\ge2$. Thus, for every $2\le s\le r$,
		$$
		\mathcal{L}^2>
		\frac{r^2}{r-1}m^2
		\ge\frac{s^2}{s-1}m^2
		=\frac{s}{s-1}\sum_{j=1}^s m^2.
		$$
		Hence $L_m$ is ample by Theorem \ref{AmpTh3}. Thus $\varepsilon(\Fe,\mathcal L,p_1,\dots,p_r)\ge m.$

		For (\ref{SC2}), let
		$$
		0<m<\min\left\{\frac{b}{r},\,a\right\}.
		$$
		Then $a>m>0$ and $b>rm$, so $L_m$ is ample by
		Theorem \ref{AmpleTh1}.

		For (\ref{SC3}), let
		$$
		0<m<
		\min\left\{
			\min_{e\le i\le\lambda}
			\frac{b+(i-e)a}{k_i},\,a
		\right\}.
		$$
		Then $a>m>0$ and
		$$
		b+(i-e)a>k_i m
		\qquad\text{for every }e\le i\le\lambda.
		$$
		Together with $b>a\lambda$, these inequalities give
		the ampleness of $L_m$ by Theorem \ref{AmpleTh2}.

		In each case, taking the supremum over positive
		rational numbers $m$ below the stated bound gives
		the required inequality.
	\end{proof}
\end{corollary}

\subsection{Application to interpolation problem}

The interpolation problem on curves in Hirzebruch surfaces $\Fe$ is a topic of recent interest (see  \cite{Laf02, HJNS,JS26}). Identifying effective non-special divisors $D$ (i.e., $h^1(\Fer, \mathcal O_{\Fer}(D))=0)$ on $\Fer$ is a way to answer the problem. Our theorems give a collection of non-special divisors.

\begin{corollary}\label{special}
    Let $D=aH_e+bF_e-m_1E_1-\dots-m_rE_r$ be a divisor on $\Fer$. Then $h^1(\Fer, \mathcal O_{\Fer}(D))=0$ if one of the following holds
    \begin{enumerate}

    \item \label{S1}
    \begin{enumerate}
       \item $a+2>m_i+1>0$ ~$\text{for } 1\le i\le r,$
			\item $b+e+2>(a+2)e,$
			\item $b+e+2>\displaystyle \sum_{i=1}^{r}(m_i+1).$
    \end{enumerate}\vspace{.5 cm}
        \item \label{S2}  Let $\lambda\ge e$ be the least integer such that
    $2\lambda-e+2>r$. Also, let $k_i=2i-e+1$ and
    $\ell_i=\min\{k_i,r\}$ for all $e\le i\le\lambda$. \vspace{.4cm}
    
    \begin{enumerate}
        \item 
        $a+2>m_i+1>0$, for $1\le i\le r$,
        \item
        $\displaystyle b+(i-e)(a+2)>\sum_{j\in I}(m_j+1)$,
        for every $e\le i\le\lambda$ and every subset
        $I\subseteq\{1,\dots,r\}$ of cardinality $\ell_i$,
        
        \item 
        $b>(a+2)\lambda$.
    \end{enumerate}\vspace{.5 cm}
    
    \item\label{S3} \begin{enumerate}
        \item $a+2>m_i+1>0$, ~~$\forall 1\le i\le r,$
		\item  $b+e+2>(a+2)e,$
		\item $\displaystyle 2(a+2)(b+e+2)-(a+2)^2e >
		\frac{s}{s-1}\sum_{j=1}^s (m_{i_j}+1)^2,$
		for every $2\le s\le r$ and every choice of distinct indices
		$i_1,i_2,\dots,i_s\in\{1,\dots,r\}$.
    \end{enumerate}
    \end{enumerate}

    Further, for $D$ satisfying any of (\ref{S1}), (\ref{S2}), (\ref{S3}), we have
    $$\dim~H^0(\Fer, \mathcal O_{\Fer}(D))=
(a+1)\left(b+1-\frac{ae}{2}\right)
-\sum_{i=1}^r\binom{m_i+1}{2}.$$
\end{corollary}
\begin{proof}
    Note that 
    $$D-K_{\Fer}=(a+2)H_e+(b+e+2)F_e-(m_1+1)E_1-\dots-(m_r+1)E_r.$$ The Coefficient satisfying any of these conditions (\ref{S1}) or (\ref{S2}) or (\ref{S3}), gives ampleness of $D-K_{\Fer}$ using Theorem \ref{AmpleTh1}, Theorem \ref{AmpleTh2} and Theorem \ref{AmpTh3} respectively. Then, immediately using  the Kodaira vanishing theorem, we can see that 
    $H^1(\Fer, \mathcal O_{\Fer}(D))=0$. Thus, $D$ is non-special.

    From Serre duality, we have 
    $$h^2(\Fer, \mathcal O_{\Fer}(D))=h^0(\Fer, \mathcal O_{\Fer}(K_{\Fer}-D))=0.$$
    The last equality follows from $a>0.$ Then, from Riemann-Roch theorem, we have 
    $$h^0(\Fer, \mathcal O_{\Fer}(D))=\frac{D^2-K_{\Fer}\cdot D}{2}+1=
(a+1)\left(b+1-\frac{ae}{2}\right)
-\sum_{i=1}^r\binom{m_i+1}{2}.$$
\end{proof}

	\section{Global Generation}\label{Global Generation}
	%Globally generated line bundles play a fundamental role in algebraic geometry4 provide morphisms from a surface to projective spaces. %Also serve as an essential tool in understanding the positivity of line bundles. 
    %In this section, we construct three sets of finite conditions to conclude the global generation of a line bundle.
	In this section, we use the conditions for ampleness proved in Section \ref{Ample} and Reider's criterion (see Theorem \ref{Reider's}) to get sufficient conditions for the global generation of a line bundle on $\Fer$.
	
	\begin{theorem}\label{GGTh1}
    Suppose $e>0$. Let $\pi : \Fer \to \Fe$ be the blow up of
    $\Fe$ at $r\ge1$ distinct points $p_1,p_2,\dots,p_r$ such that
    $p_i\notin C_e$ for every $i$, and no two points lie on the
    same fiber of the map $\phi : \Fe \to \mathbb{P}^1$.
    Let $L=aH_e+bF_e-m_1E_1-m_2E_2-\dots-m_rE_r$ be a line
    bundle on $\Fer$ such that:
    \begin{enumerate}
        \item \label{G1C1}
        $a+2>m_i+1>0$, for $1\le i\le r$,
        
        \item \label{G1C2}
        $b+e+2>(a+2)e$,
        
        \item \label{G1C3}
        $\displaystyle b+e+2>\sum_{i=1}^r(m_i+1)$.
    \end{enumerate}
    Then $L$ is globally generated.
    
    \begin{proof}
        First note that, since the coefficients are integers,
        we have the following:
        \begin{enumerate}
            \item[a)]
            $\eqref{G1C1}\implies a\ge m_i\ge0$,
            
            \item[b)]
            $\eqref{G1C2}\implies
            a(b+e+2)\ge a(a+2)e=a^2e+2ae$,
            
            \item[c)]
            $\displaystyle
            \eqref{G1C3}\implies
            a(b+e+2)\ge
            a\left(\sum_{i=1}^r m_i+r+1\right)
            \ge\sum_{i=1}^r m_i^2+ra+a$
            by (a).
        \end{enumerate}
        
        Let
        \[
            N=(a+2)H_e+(b+e+2)F_e
              -\sum_{i=1}^r(m_i+1)E_i.
        \]
        By the assumptions and Theorem \ref{AmpleTh1},
        $N$ is ample. Note that $L=K_{\Fer}+N$.
        Using (b) and (c), we obtain
        \begin{equation}\label{GG1N2}
        \begin{split}
            N^2
            &=2(a+2)(b+e+2)-(a+2)^2e
              -\sum_{i=1}^r(m_i+1)^2\\
            &=2a(b+e+2)+4(b+e+2)
              -(a+2)^2e
              -\sum_{i=1}^r m_i^2
              -2\sum_{i=1}^r m_i-r\\
              &\ge  \left(\sum_{i=1}^r m_i^2+ra+a \right)+\left(a^2e+2ae\right)+(4b+4e+8)-(a^2e+4ae+4e)\\
              &\quad -\sum_{i=1}^r m_i^2
              -2\sum_{i=1}^r m_i-r\\
            &\ge ra+a-2ae+4b+8
              -2\sum_{i=1}^r m_i-r\\
              &= ra+a-2ae+2b-2e +6 +2(b+e+1)
              -2\sum_{i=1}^r m_i-r\\
            &\ge ra+a-2ae-2e+2b+6+r
              \quad\text{(by \eqref{G1C3})}\\
            &=2(b-ae-e+2)+(r+1)a+r+2\\
            &\ge r+4\ge5.
        \end{split}
        \end{equation}
        Here the last inequality follows from $a\ge0$ and
        $b-ae-e+2\ge1$, using \eqref{G1C2}.
        
        By Reider's criterion, if $L$ is not globally generated,
        there exists a nonzero effective divisor $D$ such that
        either $D^2=-1$ and $N\cdot D=0$, or $D^2=0$ and
        $N\cdot D=1$. Since $N$ is ample, the first possibility
        cannot occur.
        
        Suppose that $D^2=0$ and $N\cdot D=1$. Since $N$ is ample,
        its intersection with every reduced irreducible curve is
        a positive integer. Thus $N\cdot D=1$ implies that $D$
        is reduced and irreducible.
        Write
        \[
            D=\alpha H_e+\beta F_e-\sum_{i=1}^r n_iE_i.
        \]
        If $D$ is not exceptional, then
        $n_i=\operatorname{mult}_{p_i}\pi(D)\ge0$.
        If $\alpha>0$, intersecting $\pi(D)$ with the fiber
        through $p_i$ gives $n_i\le\alpha$.
        Hence \eqref{G1C3} implies
        \begin{equation}\label{thm4.1eq1}
            \alpha(b+e+2)
            \ge\alpha\left(\sum_{i=1}^r(m_i+1)+1\right)
            \ge\sum_{i=1}^r(m_i+1)n_i+\alpha.
        \end{equation}
        
        We consider the following cases, using
        \cite[Corollary V.2.18]{Har}.
        
        \noindent\textbf{Case (i):}
        $\alpha>0$, and either $\alpha>1$ or $\beta>\alpha e$.\\
        In this case $\beta\ge\alpha e$. By
        \eqref{thm4.1eq1}, we have
        \[
        \begin{split}
            N\cdot D
            &=\alpha(b+e+2)+(a+2)(\beta-\alpha e)
              -\sum_{i=1}^r(m_i+1)n_i\\
            &\ge\alpha+(a+2)(\beta-\alpha e) (\text{ using } \eqref{thm4.1eq1})\\
            &\ge2.
        \end{split}
        \]
        Last inequality follows as either $\alpha\ge2$, or
        $\beta-\alpha e\ge1$; also $a+2\ge2$.
        This contradicts $N\cdot D=1$.
        
        \noindent\textbf{Case (ii):}
        $\alpha=1$ and $\beta=e$.\\
        Since $0\le n_i\le1$ and $D^2=0$, precisely $e$ of the
        multiplicities are $1$. Thus
        \[
            D=H_e+eF_e-E_{i_1}-\dots-E_{i_e}.
        \]
        If $e>r$, this case cannot occur. Otherwise,
        \eqref{G1C1} and \eqref{G1C2} give
        \[
        \begin{split}
            b+e+2
            &\ge(a+2)e+1\\
            &\ge\sum_{j=1}^e(m_{i_j}+2)+1\\
            &=\sum_{j=1}^e(m_{i_j}+1)+e+1.
        \end{split}
        \]
        Consequently,
        \[
            N\cdot D
            =b+e+2-\sum_{j=1}^e(m_{i_j}+1)
            \ge e+1\ge2,
        \]
        again a contradiction.
        
        \noindent\textbf{Case (iii):}
        $\alpha=1$ and $\beta=0$.\\
        Here $D=H_e$, since none of the blown up points lies
        on $C_e$. But $D^2=-e<0$, contradicting $D^2=0$.
        
        \noindent\textbf{Case (iv):}
        $\alpha=0$ and $\beta=1$.\\
        Since no two blown up points lie on the same fiber,
        $D=F_e$ or $D=F_e-E_i$ for some $i$.
        The condition $D^2=0$ forces $D=F_e$. Hence
        \[
            N\cdot D=a+2\ge2,
        \]
        contradicting $N\cdot D=1$.
        
        \noindent\textbf{Case (v):}
        $\alpha=\beta=0$.\\
        Since $D$ is reduced and irreducible, $D=E_i$ for some
        $i$. But then $D^2=-1$, contradicting $D^2=0$.
        
        Thus no effective divisor with $D^2=0$ and $N\cdot D=1$
        exists. By Reider's criterion, $L$ is globally generated.
    \end{proof}
\end{theorem}
	
	\begin{theorem}\label{GGTh2}
    Suppose $e>0$. Let $\pi : \Fer \to \Fe$ be the blow up of
    $\Fe$ at very general points $p_1,p_2,\dots,p_r$, where $r\ge1$.
    Let $\lambda\ge e$ be the least integer such that
    $2\lambda-e+2>r$. Also, let $k_i=2i-e+1$ and
    $\ell_i=\min\{k_i,r\}$ for all $e\le i\le\lambda$.
    Let $L=aH_e+bF_e-m_1E_1-m_2E_2-\dots-m_rE_r$ be a line
    bundle on $\Fer$ such that:
    \begin{enumerate}
        \item \label{G2C1}
        $a+2>m_i+1>0$, for $1\le i\le r$,
        
        \item \label{G2C2}
        $\displaystyle
        b+e+2+(i-e)(a+2)>\sum_{j\in I}(m_j+1)+1$,
        for every $e\le i\le\lambda$ and every subset
        $I\subseteq\{1,\dots,r\}$ of cardinality $\ell_i$,
        
        \item \label{G2C3}
        $b+e+2\ge(a+2)\lambda+1$.
    \end{enumerate}
    Then $L$ is globally generated.
    
    \begin{proof}
        Let
        \[
            N=(a+2)H_e+(b+e+2)F_e
              -\sum_{i=1}^r(m_i+1)E_i.
        \]
        By the assumptions and Theorem \ref{AmpleTh2}, $N$ is
        ample. Note that $L=K_{\Fer}+N$.
        
        Since the coefficients are integers, \eqref{G2C1} gives
        $a\ge m_i\ge0$ and $a+2\ge m_i+2$.
        Also, $k_\lambda\ge r$, so $\ell_\lambda=r$.
        Applying \eqref{G2C2} with $i=\lambda$, we obtain
        \[
            b+e+2+(\lambda-e)(a+2)
            \ge\sum_{i=1}^r(m_i+1)+2.
        \]
        Hence, using \eqref{G2C3}, we have
        \begin{equation}
        \begin{split}
            N^2
            &=2(a+2)(b+e+2)-(a+2)^2e
              -\sum_{i=1}^r(m_i+1)^2\\
            &=(a+2)\bigl[b+e+2-(a+2)\lambda\bigr]\\
            &\quad +(a+2)\bigl[b+e+2+(\lambda-e)(a+2)\bigr]
              -\sum_{i=1}^r(m_i+1)^2\\
            &\ge(a+2)
              +(a+2)\left(\sum_{i=1}^r(m_i+1)+2\right)
              -\sum_{i=1}^r(m_i+1)^2\\
            &=3(a+2)+\sum_{i=1}^r(m_i+1)(a+1-m_i)\\
            &\ge3(a+2)+\sum_{i=1}^r(m_i+1)\\
            &\ge r+6\ge5.
        \end{split}
        \end{equation}
        
        Suppose that $L$ is not globally generated.
        Since $N$ is ample and $N^2\ge5$, Reider's criterion
        gives a nonzero effective divisor $D$ with
        $D^2=0$ and $N\cdot D=1$.
        Since the intersection of $N$ with every reduced
        irreducible curve is a positive integer,
        $N\cdot D=1$ forces $D$ to be reduced and irreducible.
        Write
        \[
            D=\alpha H_e+\beta F_e-\sum_{i=1}^r n_iE_i.
        \]
        We consider the following cases.
        
        \noindent\textbf{Case (i):} $\alpha=\beta=0$.\\
        Here $D=E_i$ for some $i$, but $D^2=-1$,
        contradicting $D^2=0$.
        
        \noindent\textbf{Case (ii):}
        $\alpha=0$ and $\beta=1$.\\
        Since the points are very general, $D=F_e$ or
        $D=F_e-E_i$ for some $i$. The condition $D^2=0$
        forces $D=F_e$. Thus
        \[
            N\cdot D=a+2\ge2,
        \]
        contradicting $N\cdot D=1$.
        
        \noindent\textbf{Case (iii):}
        $\alpha=1$ and $\beta=0$.\\
        Since none of the blown up points lies on $C_e$,
        we have $D=H_e$. But $D^2=-e<0$, contradicting
        $D^2=0$.
        
        For the remaining cases, using
        \cite[Corollary V.2.18]{Har}, we have
        $\alpha>0$ and $\beta\ge\alpha e$.
        Moreover, $n_i=\operatorname{mult}_{p_i}\pi(D)\ge0$,
        and intersecting $\pi(D)$ with the fiber through
        $p_i$ gives $n_i\le\alpha$.
        
        \noindent\textbf{Case (iv):}
        $\alpha=1$ and $e\le\beta\le\lambda$.\\
        Since $0\le n_i\le1$, we may write
        \[
            D=H_e+\beta F_e-E_{i_1}-\dots-E_{i_s}.
        \]
        By Proposition \ref{dim} and the very general position
        of the points,
        \[
            s\le h^0(\Fe,\mathcal{O}_{\Fe}(C_e+\beta f))-1
              =2\beta-e+1=k_\beta.
        \]
        Thus $s\le\ell_\beta$. Choose a subset
        $I\subseteq\{1,\dots,r\}$ of cardinality $\ell_\beta$
        containing $\{i_1,\dots,i_s\}$.
        Since $m_i+1>0$, we obtain
        \[
        \begin{split}
            N\cdot D
            &=b+e+2+(\beta-e)(a+2)
              -\sum_{j=1}^s(m_{i_j}+1)\\
            &\ge b+e+2+(\beta-e)(a+2)
              -\sum_{j\in I}(m_j+1)\\
            &>1
              \quad\text{(by \eqref{G2C2}),}
        \end{split}
        \]
        a contradiction.
        
        \noindent\textbf{Case (v):}
        $\alpha>1$ or $\beta>\lambda$.\\
        By Lemma \ref{multi}, we have
        \[
            \beta+(\lambda-e)\alpha\ge\sum_{i=1}^r n_i.
        \]

        Since $D^2=0$, we have
        \[
            \sum_{i=1}^r n_i^2
            =2\alpha\beta-\alpha^2e
            \ge\alpha^2e>0.
        \]
        Thus at least one $n_i$ is positive.
        Consequently, \eqref{G2C1} gives
        \begin{equation}\label{eqnn}
            (a+2)\bigl(\beta+(\lambda-e)\alpha\bigr)
            \ge(a+2)\sum_{i=1}^r n_i
            >\sum_{i=1}^r(m_i+1)n_i.
        \end{equation}
        Therefore,
        \[
        \begin{split}
            N\cdot D
            &=\alpha(b+e+2)+(a+2)\beta
              -\alpha(a+2)e-\sum_{i=1}^r(m_i+1)n_i\\
            &=(a+2)\bigl(\beta+(\lambda-e)\alpha\bigr)
              +\alpha\bigl(b+e+2-(a+2)\lambda\bigr)
              -\sum_{i=1}^r(m_i+1)n_i\\
            &>\alpha\bigl(b+e+2-(a+2)\lambda\bigr)
              \quad\text{(by \eqref{eqnn})}\\
            &\ge\alpha\ge1
              \quad\text{(by \eqref{G2C3}),}
        \end{split}
        \]
        again a contradiction.
        
        Hence no effective divisor with $D^2=0$ and
        $N\cdot D=1$ exists. By Reider's criterion,
        $L$ is globally generated.
    \end{proof}
\end{theorem}
	
	\begin{theorem}\label{GG3}
	Suppose $e>0$. Let $p_1,p_2,\dots,p_r\in\Fe$ be very
	general points with $r\ge1$. Let
	$L=aH_e+bF_e-\sum_{i=1}^r m_iE_i$ be a line bundle on
	$\Fer$ such that:
	\begin{enumerate}
		\item \label{G3C1} $a+2>m_i+1>2$, for all
		$i=1,\dots,r$,
		\item \label{G3C2} $b+e+2>(a+2)e$,
		\item \label{G3C4}
		$\displaystyle
		2(a+2)(b+e+2)-(a+2)^2e
		>\frac{s}{s-1}\sum_{j=1}^s(m_{i_j}+1)^2,$
		for every $2\le s\le r$ and every choice of distinct
		indices $i_1,i_2,\dots,i_s\in\{1,\dots,r\}$.
	\end{enumerate}
	Then $L$ is globally generated.
	\begin{proof}

    If $r=1$, then by (1), (2) and $e\geq 1$, we have
$$b+e+2>(a+2)e\geq a+2>m_1+1.$$
Hence $L$ is globally generated by Theorem 4.1.
Thus, we may assume that $r\geq 2$.

		Let
		$$
		N=(a+2)H_e+(b+e+2)F_e-\sum_{i=1}^r(m_i+1)E_i.
		$$
		By the assumptions on $L$ and Theorem \ref{AmpTh3},
		$N$ is ample. Also, $L=K_{\Fer}+N$. Using
		\eqref{G3C4} for $s=r$, we have
		\begin{equation*}
			\begin{split}
				N^2
				&=2(a+2)(b+e+2)-(a+2)^2e
				-\sum_{i=1}^r(m_i+1)^2\\
				&>\frac{1}{r-1}\sum_{i=1}^r(m_i+1)^2\\
				&\ge\frac{9r}{r-1}>9,
			\end{split}
		\end{equation*}
		since $m_i+1\ge3$ by \eqref{G3C1}.

		Suppose that $L$ is not globally generated. Since $N$
		is ample, Reider's criterion gives an effective divisor
		$D$ with $D^2=0$ and $N\cdot D=1$. As $N$ has positive
		integral intersection with every reduced irreducible
		curve, $D$ must itself be reduced and irreducible.

		Neither an exceptional divisor nor $H_e$ has
		self-intersection zero. If $\pi(D)$ is a fiber, then
		$D^2=0$ implies $D=F_e$, and
		$N\cdot D=a+2\ge4$, a contradiction.
		Thus, using \cite[Corollary V.2.18]{Har}, we may write
		$$
		D=\alpha H_e+\beta F_e-\sum_{i=1}^r n_iE_i,
		$$
		where $\alpha>0$, $\beta\ge\alpha e$, and
		$n_i=\text{mult}_{p_i}\pi(D)\ge0$.
		Intersecting $\pi(D)$ with the fiber through $p_i$
		gives $n_i\le\alpha$. We have
		\begin{equation}\label{GG3e2}
			2\alpha\beta-\alpha^2e
			=\sum_{i=1}^r n_i^2,
		\end{equation}
		and
		\begin{equation}\label{GG3e3}
			N\cdot D
			=(a+2)\beta+\alpha(b+e+2)
			-\alpha(a+2)e-\sum_{i=1}^r(m_i+1)n_i.
		\end{equation}

		Let $s$ denote the cardinality of
		$\{i\mid n_i>0,\ 1\le i\le r\}$.
		If $s=0$, then \eqref{GG3e2} contradicts
		$$
		2\alpha\beta-\alpha^2e\ge\alpha^2e>0.
		$$
		If $s=1$, let $n_{i_1}$ be the only positive
		multiplicity. By \eqref{G3C1} and \eqref{G3C2},
		$$
		b+e+2\ge(a+2)e+1\ge a+3\ge m_{i_1}+3.
		$$
		Hence,
		\begin{equation*}
			\begin{split}
				N\cdot D
				&=(a+2)(\beta-\alpha e)
				+\alpha(b+e+2)-(m_{i_1}+1)n_{i_1}\\
				&\ge\bigl(b+e+2-(m_{i_1}+1)\bigr)n_{i_1}\\
				&\ge2n_{i_1}\ge2,
			\end{split}
		\end{equation*}
		again a contradiction.

		Now suppose that $s\ge2$, and let $i_1,\dots,i_s$
		be the indices for which $n_i>0$. Set
		$$
		T=\sum_{j=1}^s(m_{i_j}+1)n_{i_j}.
		$$
        Note that $T>0.$
		Using \eqref{Eqn4}, \eqref{GG3e2}, and
		\eqref{G3C4}, we get
		\begin{equation*}
			\begin{split}
				(N\cdot D+T)^2&=((a+2)(\beta-\alpha e)
				+\alpha(b+e+2))^2\\
				&\ge
				\bigl[2(a+2)(b+e+2)-(a+2)^2e\bigr]
				(2\alpha\beta-\alpha^2e)\\
				&>
				\frac{s}{s-1}
				\left(\sum_{j=1}^s(m_{i_j}+1)^2\right)
				\left(\sum_{j=1}^s n_{i_j}^2\right)\\
				&\ge\frac{s}{s-1}T^2
				\quad\text{(by Cauchy-Schwarz's inequality).}
			\end{split}
		\end{equation*}
		Since $T\ge3s$, we have
		$$
		\frac{s}{s-1}T^2-(T+1)^2
		=\frac{T^2}{s-1}-2T-1
		>3T-2T-1
		=T-1>0.
		$$
		Thus,
		$$
		(N\cdot D+T)^2>(T+1)^2.
		$$
		Since $N\cdot D+T>0$, this implies $N\cdot D>1$,
		a contradiction. Hence $L$ is globally generated.
	\end{proof}
\end{theorem}
	
\section{k-Very Ampleness}
	
	             %$k$-very ampleness is a stronger notion of positivity that generalizes global generation and very ampleness. Specifically, $0$-very ampleness is equivalent to global generation, while $1$-very ampleness corresponds to very ampleness.
                 
                 In this section, we provide conditions that are sufficient to determine whether a line bundle is k-very ample. The main tool used to establish these numerical criteria is Theorem~\ref{BFS}.

		\begin{theorem}\label{kveryThm1}
	Suppose $e>0$, and let $k>0$ be an integer. Let
	$\pi:\Fer\to\Fe$ be the blow up of $\Fe$ at very general
	points $p_1,p_2,\dots,p_r$. Let
	$L=aH_e+bF_e-m_1E_1-m_2E_2-\dots-m_rE_r$
	be a line bundle on $\Fer$ such that:
	\begin{enumerate}
		\item \label{K1C1} $a+2>m_i+2k>3k-1$,
		~~$\forall 1\le i\le r,$
		\item \label{K1C2} $b+e+2>(a+2)e+2k+1,$
		\item \label{K1C3}
		$\displaystyle b+e+2\ge\sum_{i=1}^r(m_i+1)+2k.$
	\end{enumerate}
	Then $L$ is $k$-very ample.
	\begin{proof}
		Let
		$$
		N=(a+2)H_e+(b+e+2)F_e
		-\sum_{i=1}^r(m_i+1)E_i.
		$$
		By the assumptions on $L$ and Theorem \ref{AmpleTh1},
		$N$ is ample. Note that $L=K_{\Fer}+N$.
		By \eqref{K1C1}, we have
		$$
		m_i\ge k,\qquad
		a+2\ge m_i+2k+1\ge3k+1.
		$$
		Using \eqref{K1C2} and \eqref{K1C3}, we get
		\[
		\begin{split}
			N^2
			&=(a+2)(b+e+2)
			+(a+2)\bigl[(b+e+2)-(a+2)e\bigr]
			-\sum_{i=1}^r(m_i+1)^2\\
			&\ge (a+2)\left(\sum_{i=1}^r(m_i+1)+2k\right)+
			(2k+2)(a+2)-\sum_{i=1}^r(m_i+1)^2
			\\
            &=2k(a+2)+\sum_{i=1}^r(m_i+1)\bigl[(a+2)-(m_i+1)\bigr]
			+(a+2)(2k+2)\\
			&\ge(3k+1)(4k+2)\\
			&\ge4k+5.
		\end{split}
		\]

		Suppose that $L$ is not $k$-very ample. Then by
		Theorem \ref{BFS}, there exists a nonzero effective
		divisor $D$ such that
		\begin{equation}\label{aaa1}
			N\cdot D-k-1\le D^2
			\le\frac{N\cdot D}{2}<k+1.
		\end{equation}
		Since $N$ is ample, we have
		$$
		1\le N\cdot D\le2k+1.
		$$

		We first estimate $N\cdot C$ for every reduced
		irreducible curve $C$ on $\Fer$, considering the
		following cases.\\

		\noindent\textbf{Case (i):} $C=E_i$ for some $i$.\\
		By \eqref{K1C1}, we have
		$$
		N\cdot C=m_i+1\ge k+1.
		$$
		Also, $C^2=-1$.\\

		\noindent\textbf{Case (ii):} $\pi(C)$ is a fiber.\\
		Since the blown up points are very general, we have
		$C=F_e$ or $C=F_e-E_i$ for some $i$.
		If $C=F_e$, then
		$$
		N\cdot C=a+2\ge3k+1>2k+1.
		$$
		If $C=F_e-E_i$, then
		$$
		N\cdot C=a+2-(m_i+1)\ge2k,
		$$
		and $C^2=-1$.\\

		\noindent\textbf{Case (iii):} $\pi(C)=C_e$.\\
		Then $C=H_e$, and by \eqref{K1C2},
		$$
		N\cdot C=b+e+2-(a+2)e>2k+1.
		$$

		For the remaining cases, write
		$$
		C=\alpha H_e+\beta F_e-\sum_{i=1}^r n_iE_i,
		$$
		where $\alpha>0$, $\beta\ge\alpha e$, and
		$n_i=\text{mult}_{p_i}\pi(C)\ge0$.
		Intersecting $\pi(C)$ with the fiber through $p_i$
		gives $n_i\le\alpha$.\\

		\noindent\textbf{Case (iv):} $\alpha=1$ and $\beta=e$.\\
		Here $n_i\le1$, so
		$$
		C=H_e+eF_e-E_{i_1}-\dots-E_{i_s}.
		$$
		Since the points are very general and
		$h^0(\Fe,C_e+ef)=e+2$, we have $s\le e+1$.
		By \eqref{K1C3},
		$$
		N\cdot C
		=b+e+2-\sum_{j=1}^s(m_{i_j}+1)\ge2k.
		$$
		If $s=e+1$, then $C^2=e-s=-1$.
		If $s\le e$, using \eqref{K1C1} and \eqref{K1C2},
		we obtain
		\[
		\begin{split}
			N\cdot C
			&\ge(a+2)e+2k+2
			-s\bigl[(a+2)-2k\bigr]\\
			&=(e-s)(a+2)+2k(s+1)+2\\
			&\ge2k+2.
		\end{split}
		\]

		\noindent\textbf{Case (v):}
		$\alpha>1$ or $\beta>\alpha e$.\\
		Using \eqref{K1C3} and $n_i\le\alpha$, we have
		\[
		\begin{split}
			N\cdot C
			&=(a+2)(\beta-\alpha e)
			+\alpha(b+e+2)-\sum_{i=1}^r(m_i+1)n_i\\
			&\ge(a+2)(\beta-\alpha e)+2\alpha k
			+\sum_{i=1}^r(m_i+1)(\alpha-n_i)\\
			&\ge(a+2)(\beta-\alpha e)+2\alpha k.
		\end{split}
		\]
		If $\alpha>1$, then $2\alpha k\ge4k\ge2k+2$.
		Otherwise, $\beta-\alpha e\ge1$, and the same lower
		bound follows. Thus $N\cdot C\ge2k+2$.

		These cases show that $N\cdot C\ge k+1$ for every
		reduced irreducible curve $C$ on $\Fer$.
		Write $D=\sum_{\ell=1}^t d_\ell C_\ell$, where
		$d_\ell$ are positive integers and $C_\ell$ are
		distinct reduced irreducible curves. Then
		$$
		2k+1\ge N\cdot D
		=\sum_{\ell=1}^t d_\ell(N\cdot C_\ell)
		\ge(k+1)\sum_{\ell=1}^t d_\ell.
		$$
		Hence $\sum_{\ell=1}^t d_\ell=1$, so $D$ is reduced
		and irreducible.

		Moreover, the cases above show that every reduced
		irreducible curve with $N\cdot C\le2k+1$ has
		self-intersection $-1$. Therefore $D^2=-1$.
		But $N\cdot D\ge k+1$, so \eqref{aaa1} gives
		$$
		0\le N\cdot D-k-1\le D^2=-1,
		$$
		a contradiction. Hence $L$ is $k$-very ample.
	\end{proof}
\end{theorem}

		\begin{theorem}
	Suppose $e>0$, and let $k>0$ be an integer. Let
	$\pi:\Fer\to\Fe$ be the blow up of $\Fe$ at very general
	points $p_1,p_2,\dots,p_r$. Let $\lambda\ge e$ be the
	least integer such that $2\lambda-e+2>r$. Also, let
	$k_i=2i-e+1$ and $\ell_i=\min\{k_i,r\}$ for every
	$e\le i\le\lambda$.
	Let $L=aH_e+bF_e-m_1E_1-m_2E_2-\dots-m_rE_r$
	be a line bundle on $\Fer$ such that:
	\begin{enumerate}
		\item \label{K2C1}
		$a+2>m_i+2k>3k-1$, for $1\le i\le r$,
		\item \label{K2C2}
		$\displaystyle
		b+e+2+(i-e)(a+2)>
		\sum_{j\in I}(m_j+1)+2k+1,$
		for every $e\le i\le\lambda$ and every subset
		$I\subseteq\{1,\dots,r\}$ of cardinality $\ell_i$,
		\item \label{K2C3}
		$b+e+2>(a+2)(\lambda+1)$.
	\end{enumerate}
	Then $L$ is $k$-very ample.
\end{theorem}

\begin{proof}
	Let
	$$
	N=(a+2)H_e+(b+e+2)F_e-\sum_{i=1}^r(m_i+1)E_i.
	$$
	By the assumptions on $L$ and Theorem \ref{AmpleTh2},
	$N$ is ample. Note that $L=K_{\Fer}+N$.
	By \eqref{K2C1}, we have
	$$
	m_i\ge k,\qquad
	a+2\ge m_i+2k+1\ge3k+1.
	$$
	Since $k_\lambda\ge r$, we have $\ell_\lambda=r$.
	Thus \eqref{K2C1} and \eqref{K2C2} give
	\begin{equation}\label{eqn5}
		\begin{split}
			(a+2)(b+e+2)
			&>(a+2)\left(
			(e-\lambda)(a+2)+\sum_{i=1}^r(m_i+1)
			\right)\\
			&\ge(a+2)^2(e-\lambda)
			+\sum_{i=1}^r(m_i+1)(m_i+2k+1).
		\end{split}
	\end{equation}
	Using \eqref{eqn5} and \eqref{K2C3}, we obtain
	\[
	\begin{split}
		N^2
		&=2(a+2)(b+e+2)-(a+2)^2e
		-\sum_{i=1}^r(m_i+1)^2\\
        &> (a+2)^2(\lambda +1)+(a+2)(b+e+2)-(a+2)^2e
		-\sum_{i=1}^r(m_i+1)^2\\
		&>(a+2)^2(\lambda+1)+(a+2)^2(e-\lambda)\\
		&\qquad
		+\sum_{i=1}^r(m_i+1)(m_i+2k+1)
		-(a+2)^2e-\sum_{i=1}^r(m_i+1)^2\\
		&=(a+2)^2+2k\sum_{i=1}^r(m_i+1)\\
		&\ge(a+2)^2\\
		&\ge(3k+1)^2>4k+5.
	\end{split}
	\]

	We first estimate $N\cdot C$ for every reduced
	irreducible curve $C$ on $\Fer$, considering the
	following cases.\\

	\noindent\textbf{Case (i):} $C=E_i$ for some $i$.\\
	By \eqref{K2C1}, we have
	$$
	N\cdot C=m_i+1\ge k+1.
	$$
	Also, $C^2=-1$.\\

	\noindent\textbf{Case (ii):} $\pi(C)$ is a fiber.\\
	Since the points are very general, $C=F_e$ or
	$C=F_e-E_i$ for some $i$. If $C=F_e$, then
	$$
	N\cdot C=a+2\ge3k+1>2k+1.
	$$
	If $C=F_e-E_i$, then
	$$
	N\cdot C=a+2-(m_i+1)\ge2k,
	$$
	and $C^2=-1$.\\

	\noindent\textbf{Case (iii):} $\pi(C)=C_e$.\\
	Then $C=H_e$. Since $\lambda\ge e$, using
	\eqref{K2C3}, we get
	$$
	N\cdot C=b+e+2-(a+2)e
	>(a+2)(\lambda+1-e)
	\ge a+2>2k+1.
	$$

	For the remaining cases, using
	\cite[Corollary V.2.18]{Har}, write
	$$
	C=\alpha H_e+\beta F_e-\sum_{i=1}^r n_iE_i,
	$$
	where $\alpha>0$, $\beta\ge\alpha e$, and
	$n_i=\text{mult}_{p_i}\pi(C)\ge0$.
	Intersecting $\pi(C)$ with the fiber through $p_i$
	gives $n_i\le\alpha$.\\

	\noindent\textbf{Case (iv):}
	$\alpha=1$ and $e\le\beta\le\lambda$.\\
	Since $n_i\le1$, we have
	$$
	C=H_e+\beta F_e-E_{i_1}-\dots-E_{i_s}.
	$$
	The points are very general, so
	$$
	s\le h^0(\Fe,C_e+\beta f)-1
	=2\beta-e+1=k_\beta.
	$$
	Thus $s\le\ell_\beta$. Choose a subset
	$I\subseteq\{1,\dots,r\}$ of cardinality $\ell_\beta$
	containing $\{i_1,\dots,i_s\}$. Since $m_i+1>0$,
	\eqref{K2C2} gives
	\[
	\begin{split}
		N\cdot C
		&=b+e+2+(\beta-e)(a+2)
		-\sum_{j=1}^s(m_{i_j}+1)\\
		&\ge b+e+2+(\beta-e)(a+2)
		-\sum_{j\in I}(m_j+1)\\
		&>2k+1.
	\end{split}
	\]

	\noindent\textbf{Case (v):}
	$\alpha>1$ or $\beta>\lambda$.\\
	By Lemma \ref{multi}, we have
	$$
	\beta+(\lambda-e)\alpha\ge\sum_{i=1}^r n_i.
	$$
	 Since
	$a+2>m_i+1$, we obtain
	\begin{equation}\label{Eqn-5}
		(a+2)\bigl(\beta+(\lambda-e)\alpha\bigr)
		\ge\sum_{i=1}^r(m_i+1)n_i.
	\end{equation}
	Consequently,
	\[
	\begin{split}
		N\cdot C
		&=(a+2)\bigl(\beta+(\lambda-e)\alpha\bigr)
		+\alpha\bigl(b+e+2-(a+2)\lambda\bigr)
		-\sum_{i=1}^r(m_i+1)n_i\\
		&\ge\alpha\bigl(b+e+2-(a+2)\lambda\bigr)
		\qquad\text{(by \eqref{Eqn-5})}\\
		&>\alpha(a+2)
		\qquad\text{(by \eqref{K2C3})}\\
		&\ge a+2>2k+1.
	\end{split}
	\]

	These cases show that $N\cdot C\ge k+1$ for every
	reduced irreducible curve $C$, and the only possible
	curves with $N\cdot C\le2k+1$ are $E_i$ and $F_e-E_i$.

	Suppose that $L$ is not $k$-very ample. By
	Theorem \ref{BFS}, there exists a nonzero effective
	divisor $D$ such that
	$$
	N\cdot D-k-1\le D^2
	<\frac{N\cdot D}{2}<k+1.
	$$
	Write $D=\sum_{\ell=1}^q d_\ell C_\ell$, where
	$d_\ell$ are positive integers and $C_\ell$ are
	distinct reduced irreducible curves. Then
	$$
	2k+1\ge N\cdot D
	=\sum_{\ell=1}^q d_\ell(N\cdot C_\ell)
	\ge(k+1)\sum_{\ell=1}^q d_\ell.
	$$
	Hence $\sum_{\ell=1}^q d_\ell=1$, so $D$ is reduced
	and irreducible. Therefore $D=E_i$ or $D=F_e-E_i$
	for some $i$, and $D^2=-1$. But $N\cdot D\ge k+1$,
	so
	$$
	0\le N\cdot D-k-1\le D^2=-1,
	$$
	a contradiction. Hence $L$ is $k$-very ample.
\end{proof}
		
%		\begin{lemma}\label{generelised Krishna Lemma}
%			Let $r\ge3, m_1,\dots,m_r\ge k$ and $n_1,\dots,n_r>0$. Suppose that $n_i\ge2$ for some $1\le i \le r$. Then for any $1\le j \le r$ the following inequality holds
%			\[
%			\left(\frac{r+3}{r+2}+2k+1\right)\left(\sum_{i=1}^rm_i^2\right)\left(\sum_{i=1}^rn_i^2-n_j\right)>\left(\sum_{i=1}^r(m_in_i)+2k+1\right)^2
%			\]
%		\end{lemma}
%		
%		\begin{proof}
%			By \cite[Lemma 2.3]{Han} we have,
%			$$\left(\frac{r+3}{r+2}\right)\left(\sum_{i=1}^rm_i^2\right)\left(\sum_{i=1}^rn_i^2-n_j\right)>\left(\sum_{i=1}^rm_in_i\right)^2$$
%			Now to prove the desired inequality it is enough to show that,
%			$$\left(2k+1\right)\left(\sum_{i=1}^rm_i^2\right)\left(\sum_{i=1}^rn_i^2-n_j\right)>2(2k+1)\sum_{i=1}^rm_in_i+(2k+1)^2$$
%			equivalently,
%			 $$\left(\sum_{i=1}^rm_i^2\right)\left(\sum_{i=1}^rn_i^2-n_j\right)>2\sum_{i=1}^rm_in_i+2k+1$$
%			 For convenience let $\displaystyle a=\sum_{i=1}^rm_i^2,b=\sum_{i=1}^rn_i^2,c=\sum_{i=1}^rm_in_i$. Thus we need to show:
%			 $$ab-2k-1>2c+an_j.$$
%			 Note that since $m_i\ge k$ and $r\ge 3$ we have $a\ge2k+1$.
%			 So, $ab-2k=a(b-1)+a-2k-1>a(b-1)$. Now it amounts to show $a(b-1)=ab-a>2c+an_j$ equivalently $a(b-n_j-1)>2c$(This true I can't prove it)
%		\end{proof}
	\begin{theorem}
	Suppose $e>0$, and let $k>0$ be an integer. Let
	$\pi:\Fer\to\Fe$ be the blow up of $\Fe$ at very general
	points $p_1,p_2,\dots,p_r$, where $r\ge 1$.
	Let $L=aH_e+bF_e-m_1E_1-m_2E_2-\dots-m_rE_r$
	be a line bundle on $\Fer$ such that:
	\begin{enumerate}
		\item \label{K3C1}
		$a+2>m_i+2k>3k-1$, ~~$\forall 1\le i\le r,$
		\item \label{K3C2}
		$b+e+2>(a+2)e+2k+1,$
		\item \label{K3C4}
		$\displaystyle
		2(a+2)(b+e+2)-(a+2)^2e
		\ge
		\left(\frac{s+3}{s+2}+2k+1\right)
		\sum_{j=1}^s(m_{i_j}+1)^2,$
		for every $2\le s\le r$ and every choice of distinct
		indices $i_1,i_2,\dots,i_s\in\{1,\dots,r\}$.
	\end{enumerate}
	Then $L$ is $k$-very ample.
\end{theorem}

\begin{proof}
	By \eqref{K3C1} and \eqref{K3C2}, we have
	$$
	m_i+1\ge k+1,\qquad a+2\ge m_i+2k+1,\qquad
	b+e+2-(a+2)e\ge2k+2.
	$$
	Let
	$$
	N=(a+2)H_e+(b+e+2)F_e-\sum_{i=1}^r(m_i+1)E_i.
	$$
	Note that $L=K_{\Fer}+N$. Since
	$$
	\frac{s+3}{s+2}+2k+1
	=2k+2+\frac{1}{s+2}>\frac{s}{s-1}
	\qquad\text{for every }s\ge2,
	$$
	the assumptions and Theorem \ref{AmpTh3} imply that
	$N$ is ample.
    
We show that $N^2\ge8(2k+1)>4k+5$.
If $r\ge2$, using \eqref{K3C4} for $s=r$, we have
\[
\begin{split}
	N^2
	&\ge\left(2k+1+\frac{1}{r+2}\right)
	\sum_{i=1}^r(m_i+1)^2\\
	&\ge(2k+1)r(k+1)^2\\
	&\ge8(2k+1).
\end{split}
\]
If $r=1$, using \eqref{K3C1}, \eqref{K3C2} and $e\ge1$,
we obtain
\[
\begin{split}
	N^2
	&=(a+2)^2e-(m_1+1)^2
	+2(a+2)\bigl(b+e+2-(a+2)e\bigr)\\
	&>2(3k+1)(2k+2)\\
	&\ge8(2k+1).
\end{split}
\]
Thus $N^2>4k+5$ for every $r\ge1$.

	We first estimate $N\cdot C$ for every reduced
	irreducible curve $C$ on $\Fer$.
	If $C=E_i$, then
	$$
	N\cdot C=m_i+1\ge k+1.
	$$
	If $\pi(C)$ is a fiber, then $C=F_e$ or $C=F_e-E_i$.
	In these cases,
	$$
	N\cdot F_e=a+2\ge3k+1>2k+1,\qquad
	N\cdot(F_e-E_i)=a+2-(m_i+1)\ge2k.
	$$
	If $\pi(C)=C_e$, then $C=H_e$, and
	$$
	N\cdot C=b+e+2-(a+2)e>2k+1.
	$$

	For the remaining curves, using
	\cite[Corollary V.2.18]{Har}, we may write
	$$
	C=\alpha H_e+\beta F_e-\sum_{i=1}^r n_iE_i,
	$$
	where $\alpha>0$, $\beta\ge\alpha e$, and
	$n_i=\text{mult}_{p_i}\pi(C)\ge0$.
	Intersecting $\pi(C)$ with the fiber through $p_i$
	gives $n_i\le\alpha$ for every $i$.
	Let $s$ denote the number of positive multiplicities,
	and let $i_1,\dots,i_s$ be their indices.
	We show that $N\cdot C>2k+1$ by considering the
	following cases.\\

	\noindent\textbf{Case (i):} $s\le1$.\\
	If $s=0$, then
	$$
	N\cdot C=(a+2)(\beta-\alpha e)+\alpha(b+e+2)
	\ge\alpha(b+e+2)>2k+1.
	$$
	If $s=1$, let $n_{i_1}$ be the only positive
	multiplicity. Since $\alpha\ge n_{i_1}$, we have
	\[
	\begin{split}
		N\cdot C
		&=(a+2)(\beta-\alpha e)+\alpha(b+e+2)
		-(m_{i_1}+1)n_{i_1}\\
		&\ge\alpha\bigl(b+e+2-(m_{i_1}+1)\bigr)\\
		&\ge b+e+2-(m_{i_1}+1)\\
		&\ge4k+2>2k+1.
	\end{split}
	\]
	Here the last inequality follows from
	$b+e+2\ge(a+2)e+2k+2\ge a+2+2k+2$
	and $a+2\ge m_{i_1}+2k+1$.\\

	\noindent\textbf{Case (ii):} $s=2$.\\
	We claim that
	$$
	b+e+2>(m_{i_1}+1)+(m_{i_2}+1)+2k+1.
	$$
	Suppose, on the contrary, that
	$$
	b+e+2\le(m_{i_1}+1)+(m_{i_2}+1)+2k+1.
	$$
	Since
	$$
	b+e+2\ge(a+2)e+2k+2
	\ge\max\{m_{i_1}+1,m_{i_2}+1\}+4k+2,
	$$
	we get
	$$
	\min\{m_{i_1}+1,m_{i_2}+1\}\ge2k+1.
	$$
	Hence,
	$$
	(m_{i_1}+1)^2+(m_{i_2}+1)^2
	\ge(2k+1)\bigl((m_{i_1}+1)+(m_{i_2}+1)\bigr).
	$$
	Also, by Cauchy-Schwarz's inequality,
	$$
	\bigl((m_{i_1}+1)+(m_{i_2}+1)\bigr)^2
	\le2\bigl((m_{i_1}+1)^2+(m_{i_2}+1)^2\bigr).
	$$
	Therefore,
	\[
	\begin{split}
		2(a+2)(b+e+2)-(a+2)^2e
		&=\frac{(b+e+2)^2-
		\bigl(b+e+2-(a+2)e\bigr)^2}{e}\\
		&\le(b+e+2)^2-
		\bigl(b+e+2-(a+2)e\bigr)^2\\
		&\le\bigl((m_{i_1}+1)+(m_{i_2}+1)+2k+1\bigr)^2
		-(2k+2)^2\\
		&=\bigl((m_{i_1}+1)+(m_{i_2}+1)\bigr)^2\\
		&\quad+2(2k+1)\bigl((m_{i_1}+1)+(m_{i_2}+1)\bigr)
		-4k-3\\
		&\le4\bigl((m_{i_1}+1)^2+(m_{i_2}+1)^2\bigr)-4k-3\\
		&<4\bigl((m_{i_1}+1)^2+(m_{i_2}+1)^2\bigr).
	\end{split}
	\]
	But \eqref{K3C4} for $s=2$ gives
	\[
	\begin{split}
		2(a+2)(b+e+2)-(a+2)^2e
		&\ge\left(2k+2+\frac14\right)
		\bigl((m_{i_1}+1)^2+(m_{i_2}+1)^2\bigr)\\
		&>4\bigl((m_{i_1}+1)^2+(m_{i_2}+1)^2\bigr),
	\end{split}
	\]
	a contradiction. Thus the claim holds, and consequently
	\[
	\begin{split}
		N\cdot C
		&=(a+2)(\beta-\alpha e)+\alpha(b+e+2)
		-(m_{i_1}+1)n_{i_1}-(m_{i_2}+1)n_{i_2}\\
		&\ge\alpha\bigl(b+e+2-(m_{i_1}+1)-(m_{i_2}+1)\bigr)\\
		&>2k+1.
	\end{split}
	\]

	\noindent\textbf{Case (iii):} $s\ge3$.\\
	By \cite[Lemma 2.2]{HJNS}, we have
	$$
	2\alpha\beta-\alpha^2e
	\ge\sum_{j=1}^s n_{i_j}^2
	-\min\{n_{i_j}\mid1\le j\le s\}.
	$$
	Since the positive multiplicities are integers,
	$$
	\sum_{j=1}^s n_{i_j}^2
	\ge s\min\{n_{i_j}\mid1\le j\le s\}.
	$$
	It follows that
	$$
	2\alpha\beta-\alpha^2e
	\ge\frac{s-1}{s}\sum_{j=1}^s n_{i_j}^2.
	$$
	Using \eqref{Eqn4}, \eqref{K3C4}, and Cauchy-Schwarz's
	inequality, we get
	\begin{equation}\label{Eqn-3}
		\begin{split}
			\left(N\cdot C+\sum_{j=1}^s(m_{i_j}+1)n_{i_j}\right)^2
			&=\bigl((a+2)\beta+(b+e+2)\alpha
			-(a+2)\alpha e\bigr)^2\\
			&\ge\bigl[2(a+2)(b+e+2)-(a+2)^2e\bigr]
			(2\alpha\beta-\alpha^2e)\\
			&\ge\left(2k+2+\frac{1}{s+2}\right)\frac{s-1}{s}
			\left(\sum_{j=1}^s(m_{i_j}+1)^2\right)
			\left(\sum_{j=1}^s n_{i_j}^2\right)\\
			&\ge\left(2k+2+\frac{1}{s+2}\right)\frac{s-1}{s}
			\left(\sum_{j=1}^s(m_{i_j}+1)n_{i_j}\right)^2.
		\end{split}
	\end{equation}
	For $s\ge3$, we have
	$$
	\left(2k+2+\frac{1}{s+2}\right)\frac{s-1}{s}
	\ge\left(4+\frac{1}{s+2}\right)\frac{s-1}{s}
	>\left(1+\frac{2}{s}\right)^2.
	$$
	Indeed,
	$$
	\left(4+\frac{1}{s+2}\right)\frac{s-1}{s}
	-\left(1+\frac{2}{s}\right)^2
	=\frac{3(s-3)^3+26(s-3)^2+54(s-3)+1}{s^2(s+2)}>0.
	$$
	Moreover,
	$$
	\sum_{j=1}^s(m_{i_j}+1)n_{i_j}\ge s(k+1),
	$$
	so
	$$
	2k+1<\frac{2}{s}\sum_{j=1}^s(m_{i_j}+1)n_{i_j}.
	$$
	Thus \eqref{Eqn-3} gives
	\[
	\begin{split}
		\left(N\cdot C+\sum_{j=1}^s(m_{i_j}+1)n_{i_j}\right)^2
		&>\left(1+\frac{2}{s}\right)^2
		\left(\sum_{j=1}^s(m_{i_j}+1)n_{i_j}\right)^2\\
        &=
		\left(\left(1+\frac{2}{s}\right)\left(\sum_{j=1}^s(m_{i_j}+1)n_{i_j}\right)\right)^2\\
		&>\left(\sum_{j=1}^s(m_{i_j}+1)n_{i_j}+2k+1\right)^2.
	\end{split}
	\]
	Since
	$$
	N\cdot C+\sum_{j=1}^s(m_{i_j}+1)n_{i_j}
	=(a+2)(\beta-\alpha e)+\alpha(b+e+2)>0,
	$$
	we conclude that $N\cdot C>2k+1$.

	These cases show that $N\cdot C\ge k+1$ for every
	reduced irreducible curve $C$, and the only possible
	curves with $N\cdot C\le2k+1$ are $E_i$ and $F_e-E_i$.

	Suppose now that $L$ is not $k$-very ample.
	By Theorem \ref{BFS}, there exists a nonzero effective
	divisor $D$ such that
	$$
	N\cdot D-k-1\le D^2
	<\frac{N\cdot D}{2}<k+1.
	$$
	Write $D=\sum_{\ell=1}^q d_\ell C_\ell$, where the
	$C_\ell$ are distinct reduced irreducible curves and
	$d_\ell$ are positive integers. Then
	$$
	2k+1\ge N\cdot D
	=\sum_{\ell=1}^q d_\ell(N\cdot C_\ell)
	\ge(k+1)\sum_{\ell=1}^q d_\ell.
	$$
	Hence $\sum_{\ell=1}^q d_\ell=1$, so $D$ is reduced
	and irreducible. Therefore $D=E_i$ or $D=F_e-E_i$
	for some $i$, and in either case $D^2=-1$.
	But $N\cdot D\ge k+1$, so
	$$
	0\le N\cdot D-k-1\le D^2=-1,
	$$
	a contradiction. Hence $L$ is $k$-very ample.
\end{proof}

	 \begin{remark}
            Each theorem presented in this section establishes a sufficient condition for very ampleness in the case \(k = 1\). However, invoking Reider's criterion (Theorem \ref{Reider's}) and employing computations similar to those carried out in Section \ref{Global Generation}, we can obtain slightly better conditions. 
        \end{remark}

\end{document}